\documentclass[a4paper,12pt]{amsart}
\usepackage{CJK}
\usepackage{romannum}
\usepackage{amsfonts}
\usepackage{mathtools}
\usepackage{amsmath,amscd}

\usepackage{ifthen}
\usepackage{amsrefs}
\usepackage{mathrsfs}
\usepackage{amsthm}
\usepackage{amssymb}

\usepackage{tikz-cd}
\usepackage{graphicx}
\usepackage{relsize}

\usepackage{MnSymbol}

\usepackage{hyperref}
\usepackage{url}
\usepackage{etoolbox}
\usepackage{rotating}

\usepackage[shortlabels]{enumitem}
\usepackage[paper=a4paper,left=20mm,right=20mm,top=25mm,bottom=30mm]{geometry}

\usepackage{tikz}

\usetikzlibrary{backgrounds}
\usepackage{tkz-euclide}
\usepackage{xcolor}
\usetikzlibrary{trees,snakes,shapes.geometric}
\usepackage[outline]{contour}
\contourlength{1.5pt}
\usetikzlibrary{positioning}
\usetikzlibrary{%
  matrix,%
  calc,%
  arrows%
}

\setlist[enumerate]{topsep=0em, itemsep= -0em, parsep = 0 em, label=$(\alph*)$}

\let\emptyset\varnothing

\nocite{*}

\newcommand{\cA}{\mathcal{A}}

\newcommand{\cC}{\mathcal{C}}
\newcommand{\cD}{\mathcal{D}}

\newcommand{\cK}{\mathcal{K}}

\newcommand{\cN}{\mathcal{N}}

\newcommand{\cS}{\mathcal{S}}

\DeclareMathOperator{\rad}{rad}
\DeclareMathOperator{\rep}{rep}

\DeclareMathOperator{\Hom}{Hom}

\DeclareMathOperator{\modd}{mod}
\DeclareMathOperator{\supp}{supp}

\DeclareMathOperator{\Irr}{Irr}
\DeclareMathOperator{\Tr}{Tr}

\let\emptyset\varnothing

\newtheorem{lemma}{Lemma}

\newtheorem{Cor}[lemma]{Corollary}
\newtheorem{proposition}{Proposition}[section]
\newtheorem{Theorem}[proposition]{Theorem}
\newtheorem{Lemma}[proposition]{Lemma}

\newtheorem{corollary}[proposition]{Corollary}

\newenvironment{example}[1][Example.]{\begin{trivlist}
\item[\hskip \labelsep {\bfseries #1}]}{\end{trivlist}}

\newenvironment{Definition}[1][Definition.]{\begin{trivlist}
\item[\hskip \labelsep {\bfseries #1}]}{\end{trivlist}}

\DeclareUrlCommand\email{\urlstyle{rm}}

\author{Jie Liu }

\title{Middle terms of AR-sequences of graded Kronecker modules} 

\address{SUSTech International Center For Mathematics,   Southern University of Science and Technology, shenzhen 518055, China }

\address{Christian-Albrechts-Universität zu Kiel, Heinrich-Hecht-Platz 6, 24118 Kiel, Germany}

\begin{document}

\rmfamily


\pagenumbering{arabic}
\thispagestyle{empty}

\maketitle

\begin{center}
 \url{hbecun@foxmail.com}
 
 Tel.:+86 0755-88011089   \hspace{2mm} Fax: +86 0755-88010224 
\end{center}

\begin{abstract}
Let $(T(n),\Omega)$ be the covering of the generalized Kronecker quiver $K(n)$, where $\Omega$ is a bipartite orientation. Then there exists a reflection functor $\sigma$ on the category $\modd(T(n),\Omega)$.  Suppose that  $0\rightarrow X\rightarrow Y\rightarrow Z\rightarrow 0$  is an AR-sequence in the regular component $\cD$ of $\modd(T(n),\Omega)$, and  $b(Z)$ is the number of flow modules in the $\sigma$-orbit of $Z$. Then the middle term $Y$ is a sink (source or flow) module if and only if $\sigma Z$ is a sink (source or flow) module. Moreover, their radii and centers satisfy $r(Y)=r(\sigma Z)+1$ and $C(Y)=C(\sigma Z)$.

\end{abstract}

\providecommand{\keywords}[1]
{
  \small	
  \textbf{\text{Keywords:}} #1
}


\section{introduction}

\vspace{.3cm}

Let $k$ be a  field.    Recall that the generalized Kronecker quiver $K(n)$ is just the quiver with two vertices $1$, $2$, and $n$ arrows $\gamma_1, \cdots, \gamma_n:$ $1\rightarrow 2$. Let $\mathcal{K}_n=kK(n)$, and let $\rep_k(K(n))$  denote the finite-dimensional representations of $K(n)$. If we use $\modd\cK_n$ to denote the category of finite-dimensional modules of $\cK_n$,  then there  exists an equivalence between categories  $\rep_k(K(n))$ and  $\modd \cK_n$.   We will frequently identify these two subcategories.

 Since the representation type of $K(n)$ is wild \cite[1.3]{Kerner}, it is hopeless to classify all the indecomposable modules.  Hence it is prevalent to investigate the module category of Kronecker quiver concerning the  invariants. If we let  $(T(n),\Omega)$ be the universal covering of $K(n)$, then one practical way is to find invariants for those  modules of $K(n)$ which can be lifted to its covering $(T(n),\Omega)$,  where $\Omega$ is a fixed bipartite orientation (see (2.2)). In fact, $T(n)$ is  an $n$-regular tree.  In general, we can take  $k$-modules of the quiver $T(n)$ as graded Kronecker modules (cf.\cite{Claus}). We write $\modd(T(n),\Omega)$  to denote the  category of  finite-dimensional graded Kronecker modules (or simply modules).  There exists a reflection (or shift) functor  $\sigma:\mod(T(n),\Omega)\rightarrow \mod(T(n),\sigma\Omega)$ at all sinks,  where $\sigma \Omega$ is the orientation under the action of $\sigma$. Similarly, we have reflection functor $\sigma'$ for $\modd \cK_n$. Using the push-down functor $\pi_\lambda:\modd(T(n),\Omega)\rightarrow \modd \cK_n$ \cite[3.2]{Bongartz},  we show that   $\pi_\lambda \circ \sigma(M)=\sigma' \circ \pi_\lambda(M)$ for all $M\in\modd(T(n),\Omega)$. Let $M\in\modd(T(n),\Omega)$ be an indecomposable module. We say that $M$ is regular, provided $\sigma^t(M)\neq 0$ for all $t\in\mathbb{Z}$.  

Let $M\in \modd(T(n),\Omega)$  be a  regular indecomposable module, and let  $T(M)$ be the minimal tree containing the set $\{y\in T(n)_0\mid M_y \neq 0\}$.    In 2018, Claus Michael Ringel  introduces three types of modules: \textit{sink module, source module} and \textit{flow module} for $T(M)$  and presents two invariants: \textit{the smallest possible  radius} and the \textit{center path}  in the $\sigma$-orbit of module $M$ (see (2.5) and Section \ref{sec4}). Moreover, a module  $Y\in\modd(T(n),\Omega)$ is said to be a sink (source, or flow) module if all its indecomposable direct summands are sink (source, or flow) modules with the same center. In the end of the paper \cite{Claus}, Ringel asks readers a question: for an AR-sequence $0\rightarrow X\rightarrow Y\rightarrow Z\rightarrow 0$ consisting of regular modules, what the precise relationship is between the middle term $Y$ and $\sigma Z$. Since we know a little about the reflection functor $\sigma$ in his paper \cite{Claus}, it is hard to answer that question. In this article, we follow his work and build a bigger \textit{ball} (see (2.5)) that could let us see clearly the properties of reflection functor $\sigma$ and finally give a solution to Ringel's problem. In particular, we can write down the functor $\sigma$ individually for every indecomposable  module.

We show that the functor $\sigma$ induces an  equivalence between $k$-linear full subcategories of $\modd(T(n),\Omega)$ and $\modd(T(n),\sigma \Omega)$  with quasi-inverse $\sigma^-$.  This conclusion guarantees the indecomposability of regular modules of the category $\modd(T(n),\Omega)$ under the action of $\sigma$. Furthermore, it allows us to consider AR-sequences in the regular components and give the clues to solve Ringel's problem directly.

In Section \ref{sec4}, we focus on  a regular indecomposable module $M$ of mod$(T(n),\Omega)$.  It is also well-known that the regular components (or AR-components) of $K(n)$ are components of type $\mathbb{Z}A_{\infty}$ (cf.\cite{Claus1}), and it is the same for the  covering $T(n)$.  Suppose that $\cD$ is a regular component of mod$(T(n),\Omega)$.  Then for an AR-sequence $0\rightarrow X\rightarrow Y\rightarrow Z\rightarrow 0$ in $\cD$, the module $Y$ is a sink (source, or flow) module if and only if $\sigma Z$ is a  sink (source, or flow) module.

\section{preliminaries}
In this section, we present a few concepts and some basic results.  For  convenience, we will give some definitions in a short way. A thorough introduction to this part can be found in  \cite[\Romannum{2}-\Romannum{7}]{Assem1}, \cite{Claus} or \cite{Daniel}. Throughout, $k$ will  denote a field.
 
\vspace{0.3cm}

 \textit{(2.1) Auslander-Reiten theory} \label{Graph} Given a finite-dimensional  algebra $\mathcal{A}$ over $k$, let  $\modd \cA$ denote the category of  finite-dimensional right $\mathcal{A}$-modules. Let $X, Y\in \modd \mathcal{A}$. We say that a homomorphism $f: X\rightarrow Y$ is  \textit{irreducible},   provided $f$ is neither a split monomorphism nor a split epimorphism, and if $f=f_1\circ f_2$, then either $f_1$ is a split epimorphism or $f_2$ is a split monomorphism. Let $\rad_{\cA}$ be the radical of  $\modd \cA$. Then we can define the set
\begin{center}
$\Irr(X,Y):=\rad_{\cA}(X,Y)/\rad^2_{\cA}(X,Y)$
\end{center}
 and call it the \textit{space of irreducible morphisms} \cite[\Romannum{4}.1.6]{Assem1}.

\vspace{0.3cm}

Let $\modd \cA^{op}$ denote the $\cA$-dual category of $\modd \cA$.  We introduce the $\cA$-dual functor 
\begin{center}
$(-)^t=$ Hom$_{\cA}(-,\cA): \modd \cA \rightarrow \modd \cA^{op}$.
\end{center}
Let $M\in \modd \cA$, and let $P_1\xrightarrow{p_1}P_0\xrightarrow{p_0} M\rightarrow 0$ be a minimal projective representation of $M$. Using functor $(-)^t$, we can get an exact sequence of left $\cA$-modules 
\begin{center}
$0 \rightarrow M^t\xrightarrow{p^t_0} P^t_0\xrightarrow{p^t_1}P^t_1\rightarrow \text{Coker } p^t_1\rightarrow 0$.
\end{center}
We define  $\Tr M:=$ Coker $p^t_1$ and call it the \textit{transpose} of $M$. When $M$ is indecomposable but not projective, we have  $\Tr (\Tr M)\cong M$ \cite[\Romannum{4}.2.1$(c)$]{Assem1}. Let $D=\Hom_k(-,k)$. The Auslanter-Reiten \textit{translation} in $\modd\cA$ is defined by the composition   of $\Tr$ and $D$, that is, $\tau=D\circ\Tr$ and $\tau^{-1}= \Tr\circ D$.

 Let $N\in \modd \cA$ be an indecomposable module. We call $N$  \textit{preinjective}  provided   $\tau^{-t} N =0$ for some  $t\in \mathbb{N}_0$. Dually, module $N$ is called \textit{preprojective} provided  $\tau^{t}N=0$  for some $t\in \mathbb{N}_0$. Finally, module  $N$ is called \textit{regular} if $N$ is not preinjective or preprojective.

In fact,  there exists no  nonzero morphism from preinjecitve modules to preprojective or regular modules. When $N$ is decomposable, we sometimes say that $N$ is still  regular,  provided there does not exist  indecomposable direct summand which is preinjective or preprojective in $N$.

\begin{Definition}
 The \textit{Auslander-Reiten quiver} $\Gamma(\mathcal{A})$ of $\mathcal{A}$  is defined as follows:
 
 \begin{enumerate}
 \item The points $\Gamma(\mathcal{A})_0$ correspond to the isomorphism classes $[X]$ of indecomposable modules $X$ in $\modd\mathcal{A}$.
 
 \item Let $[X],[Y]$ be two points in  $\Gamma(\mathcal{A})_0$. Then there are $\dim_k\Irr(X,Y)$ arrows from $[X]$ to $[Y]$ in $\Gamma(\mathcal{A})_1$.

 \end{enumerate}

\end{Definition}
Let $X, Y\in\modd \cA$ be indecomposable modules. We say that $X$ is a \textit{predecessor} of $Y$, provided there exists a directed path from $[X]$ to $[Y]$ in $\Gamma(\cA)$, i.e. a
chain of irreducible maps from $X$ to $Y$. We say that $X$ is a \textit{sucessor} of $Y$, provided there exists a directed path from $[Y]$ to $[X]$ in $\Gamma(\cA)$, i.e. a chain of irreducible maps from $Y$ to $X$.

\begin{Definition} 
A connected component $\cD$ of $\Gamma(\cA)$  is called \textit{preprojective, preinjective, or regular,} provided modules contained in  it are all preprojective, preinjective, or regular. 
\end{Definition}

We use $(\rightarrow Y)$ to denote the subset of $\Gamma(\cA)_0$ consisting of $Y$ and all its predecessors, and $(Y\rightarrow)$ to denote the subset of $\Gamma(\cA)_0$ consisting of $Y$ and all its sucessors, respectively.

Let $X$ be a    non-projective indecomposable module (or let $Y$ be a  non-injective indecomposable module). Then there exists  a uniquely  determined short exact sequence, called \textit{Auslander-Reiten sequence}, or simply \textit{AR-sequence} (or \textit{almost split sequence})

\begin{center}
$0\rightarrow Y\xrightarrow{f} \oplus^t_{i=1} M^{n_i}_i \xrightarrow{g} X\rightarrow 0$,
\end{center}
where $Y$ (or $X$) is indecomposable, modules $M_i$ are pairwise non-isomorphic and indecomposable. Suppose that $f=\begin{bmatrix}
f_1 \\
\vdots\\
f_t
\end{bmatrix}, g=\begin{bmatrix}
g_1,\cdots,g_t
\end{bmatrix},$ where $f_i=\begin{bmatrix}
f_{i_1} \\
\vdots\\
f_{in_i}
\end{bmatrix},  g_i=\begin{bmatrix}
g_{i1},\cdots,g_{in_i}
\end{bmatrix} $. Then the maps $f_{i_1}, f_{i_2}, \cdots, f_{i_{n_i}}: Y\rightarrow M_i$  and $g_{i_1}, \cdots, g_{i_{n_i}}: M_i\rightarrow X$ correspond to the bases of  Irr$(Y,M_i)$ and Irr$(M_i, X)$, respectively. We write $Y=\tau X$ (or $X=\tau^-Y$) and we denote this in $\Gamma(\mathcal{A})$ by $[Y]\dashleftarrow [X]$. Note that we sometimes don't distinguish $X$ (or $Y$) and its isomorphism class $[X]$ (or $[Y]$) in $\Gamma(\cA)$.  It is well-known that the regular components of $\Gamma(\mathcal{K}_n)$ are the type of $\mathbb{Z}A_{\infty}$ (cf.\cite{Claus1}), and they are of the following form (see FIGURE 1).

\begin{figure}[!h]

\begin{center}

\begin{tikzpicture}[very thick,scale=0.7]

                    [every node/.style={fill, circle, inner sep = 1pt}]


\def \n {8} 

\def \m {4} 

\def \translation {1} 

\def \ab {0.15} 

\def \Pab {0.6} 

\def \lcone {1} 

\def \ldist {3} 

\def \lcolor {red} 

\def \rcone {1} 

\def \rdist {3} 

\def \rcolor {red} 

\def \llcone {0} 

\def \lldist {4} 

\def \rrcone {0} 

\def \rrdist {4} 


\foreach \a in {0,...,\n}{

\foreach \b in {0,...,\m}{

    \ifthenelse{\a = \n \and \b < \m}{


   \node[color=black] ({\a,\b,5})at ({\a*2*\Pab},{\b*2*\Pab}) {$\circ$};

     }

     {

      \ifthenelse{\b = \m \and \a < \n}{

      \node[color=black] ({\a,\b}) at ({\a*2*\Pab+\Pab},{\b*2*\Pab+\Pab}) {$\circ$};

      \node[color=black] ({\a,\b,5})at ({\a*2*\Pab},{\b*2*\Pab}) {$\circ$};

      }

      {

      \ifthenelse{\b = \m \and \a = \n}

     {\node[color=black] ({\a,\b,5})at ({\a*2*\Pab},{\b*2*\Pab}) {$\circ$};}

    {\node[color=black] ({\a,\b}) at ({\a*2*\Pab+\Pab},{\b*2*\Pab+\Pab}) {$\circ$};

    \node[color=black] ({\a,\b,5})at ({\a*2*\Pab},{\b*2*\Pab}) {$\circ$};

      }

      }

      }

    }

    }

\foreach \s in {0,...,\n}{

\foreach \t in {0,...,\m}{  

 \ifthenelse{\t = \m \and \s < \n}{

    \draw[->] (\s*2*\Pab+\ab,\t*2*\Pab+\ab) to (\s*2*\Pab+\Pab-\ab,\t*2*\Pab+\Pab-\ab); 

    \draw[->] (\s*2*\Pab+\Pab+\ab,\t*2*\Pab+\Pab-\ab) to (\s*2*\Pab+2*\Pab-\ab,\t*2*\Pab+\ab); 

 }{

   \ifthenelse{\s = \n \and \t < \m}{




   }

  {

  \ifthenelse{\s = \n \and \t = \m}{

    }{

   \draw[->] (\s*2*\Pab+\ab,\t*2*\Pab+\ab) to (\s*2*\Pab+\Pab-\ab,\t*2*\Pab+\Pab-\ab); 

   \draw[->] (\s*2*\Pab+\Pab+\ab,\t*2*\Pab+\Pab+\ab) to (\s*2*\Pab+2*\Pab-\ab,\t*2*\Pab+2*\Pab-\ab);

   \draw[->] (\s*2*\Pab+\ab,\t*2*\Pab+2*\Pab-\ab) to (\s*2*\Pab+\Pab-\ab,\t*2*\Pab+\Pab+\ab); 

   \draw[->] (\s*2*\Pab+\Pab+\ab,\t*2*\Pab+\Pab-\ab) to (\s*2*\Pab+2*\Pab-\ab,\t*2*\Pab+\ab);    

   }

   }

   }

    }

    }

\draw[->] (\n*2*\Pab+\ab,\m*\Pab+2*\Pab+\Pab+\Pab+\ab) to (\n*2*\Pab+\Pab-\ab,\m*\Pab+2*\Pab+\Pab+\Pab+\Pab-\ab);

\ifthenelse{\isodd{\m}}


 { 

  \node[color=black] (Dots1) at (0,\m*\Pab+2*\Pab+\Pab) {$\cdots$};

  \node[color=black] (Dots2) at (1+\n*2*\Pab,\m*\Pab+2*\Pab+\Pab) {$\cdots$};

   \ifthenelse{\isodd{\n}}{

  \node[color=black] (Dots3) at (-0.2*\n*2*\Pab,2*\m*\Pab+3*\Pab) {$\vdots$};}

  {\node[color=black] (Dots4) at (0.75*\n*2*\Pab,2*\m*\Pab+2*\Pab) {$\vdots$};} 

  }


  {

  \node[color=black] (Dots1) at (0,\m*\Pab+\Pab) {$\cdots$};

  \node[color=black] (Dots2) at (1+\n*2*\Pab,\m*\Pab+\Pab) {$\cdots$};

  \ifthenelse{\isodd{\n}}{

  \node[color=black] (Dots3) at (0*\n*2*\Pab,2*\m*\Pab+3*\Pab) {$\vdots$};}

  {\node[color=black] (Dots4) at (0.25*\n*2*\Pab,2*\m*\Pab+2*\Pab) {$\vdots$};}

  }

\ifthenelse{\translation = 1}{

   \foreach \s in {0,...,\n}{

   \foreach \t in {0,...,\m}{ 

   \ifthenelse{\s = 0}{}{

      \ifthenelse{\s = \n}{\draw[->,dotted,thin] (\s*2*\Pab-\ab,\t*2*\Pab) to (\s*2*\Pab-2*\Pab+\ab,\t*2*\Pab); }{

   \draw[->,dotted,thin] (\s*2*\Pab-\ab,\t*2*\Pab) to (\s*2*\Pab-2*\Pab+\ab,\t*2*\Pab); 

   \draw[->,dotted,thin] (\s*2*\Pab-\ab+\Pab,\t*2*\Pab+\Pab) to (\s*2*\Pab-2*\Pab+\Pab+\ab,\t*2*\Pab+\Pab); 

\draw[->,dotted,thin] (\n*2*\Pab-\ab+\Pab,\t*2*\Pab+\Pab) to (\n*2*\Pab-2*\Pab+\Pab+\ab,\t*2*\Pab+\Pab); 

   }

   }}

}}

{}  

                                                                                                                                                                                             \end{tikzpicture}

\end{center}

\caption{Regular component $\mathbb{Z}A_{\infty}.$}

\label{Fig:RCGM}

\end{figure}
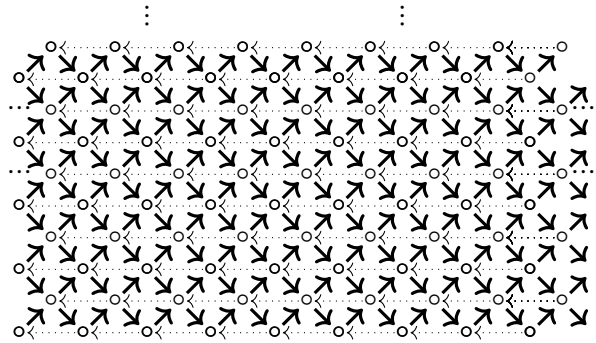

  A module $M$  is called \textit{ quasi-simple} if it is located in the bottom row of  a $\mathbb{Z}A_{\infty}$ component.  Let  $M$  be a module in a regular $\mathbb{Z}A_{\infty}$-component  $\cC$. Then there exists a  quasi-simple module $X\in \cC$ (or $Y\in \cC$)  such that  we can find a   chain of irreducible monomorphisms $X=X_1\rightarrow \cdots \rightarrow X_{s-1}\rightarrow X_s=M$ (or irreducible epimorphisms $M=Y_s\rightarrow Y_{s-1}\rightarrow \cdots \rightarrow Y_1=Y$). We call $s$ the \textit{quasi-length} of $M$ and the module $X$ (or $Y$) \textit{quasi-socle} (or \textit{quasi-top} ) of $M$. We call the factors $X_i/ X_{i-1}$  or $Y_i/Y_{i-1}$ the \textit{quasi-composition factors} of $M$. It can be shown that  $M$ is uniquely determined by its quasi-length and quasi-socle (or  quasi-top), whence we can define $ql(M):= s=$ quasi-length of $M$.

\vspace{0.3cm}	
	
\textit{(2.2) Graphs} \label{definition}	A \textit{graph} is a pair $G=(G_0, G_1):$   $G_0$ (whose elements called the vertices) and $G_1$ (whose elements called the edges). If $\{a, b\}$ is an edge, then $a,b$ are  called \textit{neighbours}.

For a graph $G$, there exsits a map $\Omega$, called \textit{orientation}  $:G \rightarrow G_0\times G_0$, such that $\Omega(\{a, a'\})$ is either ($a, a'$) or ($a', a$). We call $(G, \Omega)$  an \textit{oriented graph}. We  write $a\rightarrow a'$ if $\Omega(\{a, a'\})=(a, a')$ and call $a$ the \textit{start}, $a'$ the \textit{target}. Meanwhile we call $a$   a \textit{sink} (or a \textit{source}) if it is not a start (or the target, respectively) of any arrow. If any vertex is a sink or a source, then the orientation  $\Omega$ will be called \textit{bipartite}. A \textit{subgraph}    of a  graph $G=(G_0,G_1)$ is a graph  $G'=(G'_0,G'_1)$ such that $G'_0\subseteq G_0, G'_1\subseteq G_1$.

A  \textit{path} of length $t\geq 0$ in a graph $G$ is a finite sequence $(a_0, a_1,\cdots,a_t)$ of vertices such that   $a_{i-1}, a_i$ are neighbours for $1\leq i\leq t$, and $a_{j-1}\neq a_{j+1}$ for $1\leq j < t$. If $t=2r$, then $a_r$ is called its \textit{center} and $r$ its \textit{radius} \cite[2.1]{Claus}. If $t=2r+1$, then $\{a_r, a_{r+1}\}$ is called its \textit{center} and $r$ its \textit{radius}. Suppose that for any pair $a, b$ of vertices in $G$, there always exists a path connecting $a$ and $b$. Then  $G$ is said to be \textit{connected}. We call a path of length $t \geq 0$  a \textit{cycle} whenever its source and target coincide. A cycle of length $1$ is said to be a \textit{loop}. A graph $G$ is called \textit{acyclic} if it contains no cycles,  and it is called \textit{finite} if $G_0$ and $G_1$ are finite sets.

A graph $G$ is said to be a \textit{tree} if it is connected and paths connecting a vertex with itself are always length $0$. Let $G$ be a tree. Then  we define

\begin{center}
$d(a, b):=$ the length of a path connecting $a$ and $b$.
\end{center} 
This is well-defined since there exists only one path from $a$ to $b$ in the tree $G$. If $G'$ is a subset of tree $G$ and $x\in G$, we define

\begin{center}
$d(x, G'): =$min$\{d(x, a)\mid a \in G'\}$.
\end{center}
For every finite tree $G$, the paths of maximal possible length $d(G)$ are  called  \textit{diameters} of $G$. Let $r(G)$ denote the radius of $G$. It can be shown that for a finite tree $G$, all diameter paths of $G$ have the same center $C(G)$ \cite[2.1]{Claus}. Suppose that $(a_0,a_1,\cdots,a_t)$ is a diameter path of $G$.  We  define, 
\begin{center}
\[ C(G)=\begin{cases} 
     a_r & t=2r , \\
     \{a_r,a_{r+1}\} & t=2r+1.
   \end{cases}
\]
\end{center}

\vspace{0.3cm}

 \textit{(2.3) Quivers} A \textit{quiver} is just an oriented graph (loops and multiple arrows are allowed), usually denoted by  $Q=(Q_0, Q_1, s, t)$, where $s,t: Q_1\rightarrow Q_0$. For every $x\in Q_0$, we define a set  $\mathcal{N}(x):=\{y\in Q_0\mid \exists \alpha \in Q_1: t(\alpha)=y \text{ when } s(\alpha)=x \text{ or } s(\alpha)=y \text{ when } t(\alpha)=x\}$  and call it the \textit{neighbourhood} of $x$. When $\mid \mathcal{N}(x)\mid \leq 1$,  We say that $x$  is a  \textit{leaf}. A finite-dimensional representation $M=((M_x)_{x\in Q_0}, (M(\alpha))_{\alpha\in Q_1})$ over $Q$ consists of vector spaces $M_{x}$ and $k$-linear maps $M(\alpha): M_{s(\alpha)} \rightarrow M_{t(\alpha)}$ such that $\dim_{k}M:=\sum _{x\in Q_0}\dim_k M_{x}$ is finite. We thus define a category Rep$_k(Q)$  of  $k$-linear representations  of $Q$.  We denote by $\rep_k(Q)$ the full  subcategory of Rep$_k(Q)$ consisting of finite-dimensional representations.

Let $Q=(Q_0, Q_1, s, t)$ be a finite, connected, and acyclic quiver \cite[\Romannum{7}.5]{Assem1},  and let $n=\vert Q_0 \vert$. For every point $a\in Q_0$, we define a new quiver $\sigma_a Q=(Q'_0, Q'_1, s', t')$
 as follows: all the arrows of $Q$ having a as source or as  target are reversed, all other arrows remain unchanged. Hence there exists  a  reflection functor $\cS^+_a$ (or $\cS^{-}_a$) : rep$_k(Q)\rightarrow$ rep$_k(Q')$ when $a$ is a sink (or source). In fact, the functors $\cS^+_a$ and $\cS^-_a$ induce  equivalences between the $k$-linear full subcategories of $\modd(kQ)$ and $\modd(kQ')$ consisting of regular modules \cite[\Romannum{7}. Corollary 5.7]{Assem1}. In particular, functors $\cS^{+}_a$ and $\cS^-_a$ send regular indecomposable modules to regular indecomposable modules.  If $M\in$ mod$(kQ)$ is a regular indecomposable module, then  $\cS^-_a \cS^+_a(M)\cong M$.

\vspace{0.3cm}

\textit{(2.4) Push-down functor and graded Kronecker modules} Now we focus on the generalized Kronecker quiver $K(n)$ ($n\geq 3$). 
\[
\begin{tikzcd}
    1 \circ
    \arrow[r, draw=none, "{\vdots}" description]
    \arrow[r, bend left,        "\gamma_1"]
    \arrow[r, bend right, swap, "\gamma_n"]
    &
    \circ 2.
\end{tikzcd}
\]

Let $ \mathcal{K}_n=kK(n)$.   Let $\epsilon_i$ correspond to the trivial path point $i$, $i\in\{1, 2\}.$ In the following, we give a brief  construction of the covering $T(n)$ of Kronecker quiver $K(n)$ \cite[7.1]{Daniel}. Let $Q=(Q_0,Q_1,s,t)$ be a quiver.  We define formal inverse $(Q_1)^{-1}$ on $Q_1$, where

\begin{enumerate}
\item  $(Q_1)^{-1}:=\{\alpha^{-1}\mid \alpha \in Q_1\}$;

\item $s(\alpha^{-1}):=t(\alpha)$ and $t(\alpha^{-1}):=s(\alpha)$ for any $\alpha^{-1}\in (Q_1)^{-1}$;

\item  we say that $\omega$  is a \textit{walk} of $Q_1$, provided there exists $\omega=\alpha^{\varepsilon_n}_n\cdots \alpha^{\varepsilon_1}_1$, where $\alpha_i\in Q_1$, $\varepsilon\in\{1, -1\}$ and $s(\alpha^{\varepsilon_{i+1}}_{i+1})=t(\alpha^{\varepsilon_i}_i)$ for all $i<n$;
\item $s(\omega):=s(\alpha^{\varepsilon_1}_1)$ and $t(\omega):=t(\alpha^{\varepsilon_n}_n)$.

\end{enumerate}

Let $W:=\{\text{walks of $K(n)$}\}$. We introduce an equivalence relation $\sim$  on  $W$, which is generated by $\gamma^{-1}_i \gamma_i\sim \epsilon_1$ and $\gamma_i\gamma^{-1}_i\sim \epsilon_2$. We define an involution $(-)^{-1}$ on $W$, where 
\begin{center}
$(-)^{-1}$: $W\rightarrow W; (\alpha^{\varepsilon_n}_n\cdots \alpha^{\varepsilon_1}_1)^{-1} \mapsto \alpha^{-\varepsilon_1}_1\cdots \alpha^{-\varepsilon_n}_n$.
\end{center}
 Let $\pi(K(n))$  be the fundamental group of $K(n)$ in the point 1, i.e. $\pi(K(n))=\{[\alpha]\mid s(\alpha)=t(\alpha)=1\}$, where $[\alpha]$ is the equivalence classe of unoriented path $\alpha$. The multiplication is  the concatenation of paths. We define [$\omega$]$^{-1}$:=[$\omega^{-1}$] and the identity element is [$\epsilon_1$].  Then the quiver $T(n)$ is given by the following data:
\begin{enumerate}

\item  $(T(n))_0$ is the set of equivalence classes of walks starting in 1.

\item When $\omega^{'}\sim \gamma_i \omega$ for some $i\in\{1, \cdots, r\}$, we say that there exists an arrow from [$\omega$] to [$\omega^{'}$]. 
\end{enumerate}

Let $\pi: T(n)\rightarrow K(n); [\omega] \mapsto t(\omega), ([\omega]\rightarrow [\gamma_i \omega]) \mapsto \gamma_i$.  Action of group $G=\pi(K(n))$ on $T(n)$ is the concatenation of paths: 
\begin{center}
$g.[v]=[v\omega^{-1}]$ and
$g.([u]\rightarrow[\gamma_iu])=([u\omega^{-1}]\rightarrow [\gamma_i u\omega^{-1}])$,
\end{center}
where  $g=[\omega]\in\pi(K(n))$ and $[v], [u]\in T(n)_0$ with arrow $[u]\rightarrow [\gamma_i u]=[v]$.
Finally, we define $T^+_{n}:=\pi^{-1}(\{1\})$, $T^-_n:=\pi^{-1}(\{2\})$. Let $\Omega$ be the bipartite orientation on $T(n)$.  Then for   $M\in$ rep$_k(T(n),\Omega)$, the group $G$ acts on $M$ via:
\begin{center}
$M^g:=(((M^g)_x)_{x\in T(n)_0}$, $(M^g(\alpha))_{\alpha\in T(n)_1})$,
\end{center}
where $(M^g)_x:= M_{g.x}$ and $M^g(\alpha):=M(g.\alpha)$.

Now we fix the bipartite orientation $\Omega$ for $T(n)$, and  we define a push-down functor     $\pi_\lambda:$ rep$_k(T(n),\Omega)\rightarrow$ rep$_k(K(n)); M \mapsto (\pi_\lambda(M)_j, \pi_\lambda(M)(\gamma_i))$ factoring through $T(n)/G$, that is, $\pi_\lambda(M^g)=\pi_\lambda(M)$, $\forall g\in G$, where

\begin{equation*} 
\pi_\lambda(M)_j:=\bigoplus _{\pi(y)=j}M_y,
\end{equation*}

\begin{equation*} 
\pi_\lambda (M)(\gamma_i):=\bigoplus _{\pi(\beta)=\gamma_i} M(\beta):\pi_\lambda (M)_1\rightarrow \pi_\lambda (M)_2, j\in \{1, 2\},1\leq i\leq n.
\end{equation*}
 Note that  $\rep_k(T(n),\Omega)$ also admits AR-sequences and  the  functor $\pi_{\lambda}$ is  exact  (cf.\cite[2.2]{Bongartz}). 

 For any two sinks $x,y\in T(n)_0$, we can see that $\sigma_x \sigma_y=\sigma_y\sigma_x$ by the definition.   Then  we can define a reflection functor $\sigma:\rep_k(T(n),\Omega)\rightarrow\rep_k(T(n),\sigma\Omega)$ as the composition of the $\sigma_x$, for all sinks $x$. The functor $\sigma$ is independent of the order used in the composition, so that $\sigma$ is  well defined. We sometimes call such functor $\sigma$ the \textit{shift functor}.  By \cite[2.6]{Claus}, the composition $\sigma^2$ is the  translation  $\tau$.

 We call modules in $\rep_k(T(n),\Omega)$  the \textit{graded Kronecker modules} (or \textit{graded modules}, or simply \textit{modules}).  We often use  mod$(T(n),\Omega)$  to denote  those modules. Without specific emphasis,  we always mean quiver $(T(n),\Omega)$ when we talk about the tree $T(n)$. An indecomposable  module   $M\in\modd(T(n),\Omega)$ is said to be \textit{regular}  if  $\sigma^t(M)\neq 0$ for any $t\in \mathbb{Z}$. Let $\sigma'=\cS^+_2$ and $\sigma'^-=\cS^-_1$ for $K(n)$. 

\begin{lemma}\label{comm}
Let $M\in\modd (T(n),\Omega)$ be an indecomposable module. Then $\pi_\lambda( \sigma M)=\sigma'(\pi_\lambda(M))$ and $\pi_\lambda(\sigma^- M)=\sigma'^-(\pi_\lambda(M))$.

\end{lemma}
\begin{proof}
We only prove $\pi_\lambda(\sigma M)=\sigma'(\pi_\lambda(M))$, and it is similar for the other one. If $\sigma M=0$, then $\pi_{\lambda}(M)$ is isomorphic to a projective module, whence there is nothing to prove. Suppose that $\sigma M\neq 0$. We need to show that the following diagram commutes.

\begin{center}
\begin{tikzcd}
\text{mod}(T(n),\Omega) \arrow{r}{\pi_\lambda} \arrow{d}{\sigma}
&\text{mod } \mathcal{K}_n\arrow{d}{\sigma’}\\
\text{mod} (T(n),\sigma\Omega) \arrow{r}{\pi_\lambda} & \text{mod } \mathcal{K}_n
\end{tikzcd}
\end{center}
Following the definition,  we have

\begin{equation*} 
\pi_\lambda(M)_j=\bigoplus _{\pi(y)=j}M_y
\end{equation*}

\begin{equation*} 
\pi_\lambda (M)(\gamma_i)=\bigoplus _{\pi(\beta)=\gamma_i} M(\beta):\pi_\lambda (M)_1\rightarrow \pi_\lambda (M)_2, 
\end{equation*}
for $j\in\{1,2\}$ and $1\leq i\leq n$. Let $N=\sigma \pi_{\lambda}(M)$,  where

 \begin{equation*} 
N_2=\pi_\lambda (M)_1=\bigoplus _{\pi(y)=1}M_y, 
\end{equation*}
and 

 \begin{equation*} 
\begin{split}
N_1& = (\sigma \pi_{\lambda}(M))_1 \\
 & = \text{ker}\ (\bigoplus_{\gamma_j}\pi_{\lambda}(M)_{s(\gamma_j)}\rightarrow \pi_{\lambda} (M)_2)\\ & =\text{ker}\ (\bigoplus_{\gamma_j}\bigoplus_{\pi(\beta)=\gamma_j}M_{s(\beta)} \rightarrow  \bigoplus_{\pi(z)=2}M_z).
\end{split}
\end{equation*}
 Each $N(\gamma_i)$ is the composition of the inclusion of $N_1$ into  $\bigoplus_{\gamma_j}\pi_{\lambda}(M)_{s(\gamma_j)}$ with the projection onto the direct summand  $\pi_{\lambda}(M)_{s(\gamma_i)}$ \cite[\Romannum{7}.5.5.]{Assem1}.  On the other hand, let $N'=\sigma'(M)$ and $N''=\pi_\lambda(N')$.  Note that, 

\begin{equation*} 
N''_2=\bigoplus _{\pi(y')=2}N'_{y'}=\bigoplus _{\pi(y)=1}M_y=N_2,
\end{equation*}
and

\begin{equation*} 
\begin{split}
N''_1  & = \bigoplus _{\pi(y')=1}N'_{y'} \\ & = \bigoplus _{\pi(y')=1}{\sigma(M)}_{y'} \\ & = \bigoplus_{\pi(z)=2} \text{ker} \ (\bigoplus_{\beta:s(\beta)\rightarrow z}M_{s(\beta)} \rightarrow M_z)\\& = \text{ker}\ (\bigoplus_{\gamma_j}\bigoplus_{\pi(\beta)=\gamma_j}M_{s(\beta)} \rightarrow  \bigoplus_{\pi(z)=2}M_z)  \\
 & =\text{ker}\ (\bigoplus_{\gamma_j}\pi_{\lambda}(M)_{s(\gamma_j)}\rightarrow \pi_{\lambda} (M)_2)\\&=N_1.
\end{split}
\end{equation*}
Each $N''(\gamma_i)$ is  the composition  of inclusion of $N''_1=N_1$ into  $\bigoplus_{\gamma_j}\pi_{\lambda}(M)_{s(\gamma_j)}$ with the projection onto the direct summand  $\pi_{\lambda}(M)_{s(\gamma_i)}$. Then we know that 
\begin{center}
$N({\gamma_i})=N''({\gamma_i})$, $1\leq i\leq n$.
\end{center}
Hence $N= N''$.

\end{proof}

\begin{lemma}\label{S1}
Functors $\sigma$ and $\sigma^- $ send regular indecomposable modules to regular indecomposable modules.
\end{lemma}

\begin{proof}
Let $M\in \modd(T(n),\Omega)$ be a regular indecomposable module.     According to  \cite[6.3]{Claus4}, we can see that  $\sigma'(\pi_\lambda(M))$ is regular  indecomposable. However, Lemma \ref{comm} tells us that $\pi_\lambda ( \sigma M)=\sigma' ( \pi_\lambda(M))$, whence $\sigma M$ is regular indecomposable. For the functor $\sigma^-$, the proof is similar.
\end{proof}

\vspace{0.3cm}

\textit{(2.5) Index and Balls}\label{Ball}
 Let $ \emptyset\neq S\subseteq T(n)_0$ be a set of vertices, and let $M\in\modd(T(n),\Omega)$.
 \begin{enumerate}
 \item[•] The unique minimal tree in $T(n)$ containing $S$ is denoted by $T(S)$.
 
 \item[•] A vertex $x\in T(S)_0$ is said to be  a \textit{leaf} of $T(S)$ if $\mid \mathcal{N}(x)\cap T(S)_0\mid \leq 1$.
 
 \item[•]  The set $\supp(M):=\{y\in T(n)_0\mid M_y\neq (0)\}$ is called the \textit{support} of $M$, and the vertex $x$ is said to be a \textit{leaf} $M$ if $x$ is a leaf of $T(M):=T(\text{supp}(M))$. 
 
 \end{enumerate}

 Let $M\in \modd(T(n),\Omega)$ be an indecomposable module. Since $\dim_k M<\infty$,  $T(M)=T(\supp(M))$ is a finite tree. Using the fact that all diameter paths of a finite tree have the same center (cf.\cite{Claus}), we can  define

\begin{center}
$d(M):=$ the length of diameter paths of $M$,  

$C(M):=$ the center of diameter paths of $M$,

$r(M):=$ the radius of  diameter paths of $ M$.
\end{center}
We call them  the \textit{diameter},  \textit{center},  and  \textit{radius}  of $M$, respectively  \cite[2.4]{Claus}. Actually, we have $d(M)=d(T(M))$,  $C(M)=C(T(M))$ and $r(M)=r(T(M))$. An indecomposable module $M$ is said to be a \textit{sink} (\textit{source})  module if a diameter path (and thus all) of $T(M)$  starts and ends in sinks (sources). Otherwise we call  module $M$ a \textit{flow} module. For any $ x, c, c_1, c_2\in T(n)_0$ and $r\in \mathbb{N}$ \cite[2.3]{Claus}, let $B_r(c)=\{x\in T(n)_0 \mid d(x,c) \leq r\}$, $B_r(c_1,c_2)=\{x \in T(n)_0\mid d(x, c_i)\leq r\}$, $i\in \{1,2\}$.
We call such sets $B_r(c)$ and $B_r(c_1, c_2)$ the \textit{balls} with radius $r$ and with center $c$ or $\{c_1, c_2\}$, respectively.

We now recall some important results in \cite{Claus}.

\begin{Theorem} \label{Theorem 1}\cite[Theorem 1]{Claus}
Let $M$ be a regular indecomposable module of mod$(T(n),\Omega)$. Then the shift orbit of $M$ contains a unique sink module $M_0$ with smallest possible radius, saying with radius $r_0=r_0(M)$. Let $M_i=\sigma^i M_0$ for all $i\in\mathbb{Z}$. Then there is an integer $0\leq b\leq r_0(M)$ and a path $(a_0,\cdots, a_b)$ in $T(n)$ with the following properties:
\begin{itemize}

\item[(1)] For $i\geq 0$, the module $M_{-i}$ is a sink module with radius $r_0+i$ and center $a_0$.

\item[(2)] For $1\leq i\leq b(M)$, the module $M_i$ is a flow module with radius $r_0-1$ and center $\{a_{i-1},a_i\}$.

\item[(3)] For $i\geq 0$, the module $M_{b+1+i}$ is a source module with radius $r_0+i$ and center $a_b$.
\end{itemize}
The integer $r_0(M)$ is positive. If $r_0(M)$ is even, then $a_0$ is a sink, otherwise a source.

\end{Theorem}

Let  $M\in\modd(T(n),\Omega)$ be a regular indecomposable module. According to Theorem \ref{Theorem 1},   all the modules $M_i$ with $i\leq 0$ are sink modules. Moreover, we define its \textit{index} $\iota(M):=t$ when $M=\sigma^t M_0, t\in \mathbb{Z}$.

\begin{Theorem}\cite[Theorem 3]{Claus} \label{Theorem 3}
If $M$, $M'\in$ mod$(T(n),\Omega)$ are regular indecomposable modules and $M\rightarrow M'$ is an irreducible map, then $\iota(M')=\iota (M)-1$.
\end{Theorem}

Note that the regular Auslander-Reiten components of the category $\modd(T(n),\Omega)$ are of the form $\mathbb{Z}A_{\infty}.$ We now introduce an operator $\eta$ on the set of isomorphism classes of regular indecomposable modules.  If $\cD$ is a regular component of $\modd(T(n),\Omega)$ and $0\rightarrow X\rightarrow Y\oplus Y'\rightarrow Z\rightarrow 0$ is an Auslander-Reiten sequence in $\cD$ with $Y,Y'$  indecomposable and $r(Y)<r(Y')$, then we define $\eta(Y)=Y'$. 

\begin{Theorem}\label{Theorem 4}

\cite[Theorem 4]{Claus} Let $Y\in$ mod$(T(n),\Omega)$ be an indecomposable regular module. Then $Y$ is a sink module (or a flow module, or a source module), if and only if $\eta Y$ is a sink module (or a flow module, or a source module, respectively). Also, $\eta Y$ has the same center as $Y$  and  $r(\eta Y)=r(Y)+2$.
\end{Theorem}

\begin{proposition} \label{m sink}\cite[3.1]{Claus}
 Let $M\in$ mod$(T(n),\Omega)$ be a sink module with center $c$ , then  
 
 $d(M)-2\leqslant d(\sigma M)\leqslant d(M)$.

   There are following three possibilities:
 \begin{enumerate}

  \item[(1)] $d(\sigma M)=d(M)-2$ if and only if $\sigma M$ is a sink module. In this case, the center of $\sigma M$ is $c$. 
   
\item[(2)] $d(\sigma M)=d(M)-1$ if and only if $\sigma M$ is a flow module. The center of $\sigma M$ is of the form $\{c, c^{'}\}$ with a path $(x_0, x_1, \cdots , x_{r})$, where $x_0=c, x_1=c^{'}$ such that $x_{r}$  is a source for $\sigma \Omega$.
   
 \item[(3)] $d(\sigma M)=d(M)$ if and only if $\sigma M$ is a source module. In this case, the center of $\sigma M$ is $c$ again.
   
\end{enumerate}     

\end{proposition}

\begin{proposition}\label{Flow}

                                                                   \cite[3.2]{Claus} Let $M\in$ mod$(T(n),\Omega)$ be a flow module . We assume that $(a_{0}, \cdots, a_{d})$ with $d=2r+1$ is a diameter path of $M$ with $a_0$ a sink and $a_d$ a source. Then
\begin{center}
$d(M)\leqslant d(\sigma M)\leqslant d(M)+1.$

\end{center}

There are the following two possibilities:
\begin{enumerate}

\item[(1)] $d(M)=d(\sigma M)$, thus $\sigma M$ is a flow module. Let $a_{d+1}$ be a neighbour of $a_d$ different from $a_{d-1}$. Then $(a_1, \cdots, a_{d+1})$ is a diameter path for $\sigma M$. In particular, the center of $\sigma M$ is $\{a_{r+1}, a_{r+2}\}$ (and we note that $a_1$ is a sink for $\sigma \Omega$ and $a_{d+1}$ is a source for $\sigma \Omega$).

\item[(2)] $d(\sigma M)=d(M)+1$, then $\sigma M$ is a source module and the center of $\sigma M$ is $a_{r+1}.$

\end{enumerate}
\end{proposition}

\begin{proposition} \label{3.3}

\cite[3.3]{Claus} Let $M\in$ mod$(T(n),\Omega)$ be a source module with center $c$ and radius $r$. Then $\sigma M$ is a source module with center $c$ and radius $r+1.$
\end{proposition}

 Let $M\in \modd(T(n),\Omega)$ be an indecomposable module. Suppose that  $c$, or $\{c_1,c_2\}$, and $r$ are the center and radius of $M$, respectively. We define

\begin{center}
$B(M):=B_r(c)$, or $B(M):=B_r(c_1,c_2)$.
\end{center} 
Clearly, sets $T(M)$ and $B(M)$ have the same radius and same center. On the other hand, we can introduce a new set $B_{+1}(M)=\{x\in (T(n),\Omega)_0\mid d(x,B(M))\leq 1\}$.  Then $B_{+1}(M)$ is the ball having the same center with $B(M)$, and its radius $r_{+1}=r+1$, where $r$ is the radius of $B(M)$.  In particular,   these balls $B(M)$, $B_{+1}(M)$ can be  naturally seen as  subquivers of $(T(n),\Omega)$, whence we sometimes call them subquivers $B(M)$, $B_{+1}(M)$. Such subquivers  are finite.  

From now on, we always assume that $M$ is regular indecomposable when we talk about a module $M\in$ mod$(T(n),\Omega)$ in this section. Note that  $d(M)\geq 1$.  Let $1,\cdots, n_M$ be  all sink points in $B_{+1}(M)$, and let $1',\cdots, n'_M$ be  all source points in $B_{+1}(M)$, respectively,  $n_M, n'_M\in \mathbb{N}$.

\begin{lemma}\label{FK}
Let $M\in\modd (T(n),\Omega)$ be a regular indecomposable module. Then 
\begin{center}
$\sigma (M)=\cS^+_1\cdots \cS^+_{n_M}(M)$.
\end{center}

\end{lemma}
\begin{proof}
Let  $x\in T(n)_0\setminus B_{+1}(M)$ be a sink point.  Then  $d(x,B(M))\geq 2$,  whence we have $y\nin$ supp$(M)$ for any arrow $\alpha: y \rightarrow x$. Then $\cN(x)\cap T(M)_0=\emptyset$, whence $\cS^+_x(M)=M$. On the other hand, reflection functors $\sigma_{x_1}$ and $\sigma_{x_2}$ commute for any two different sink points $x_1,x_2\in T(n)_0$. Hence we have $\sigma(M)=\cS^+_1\cdots \cS^+_{n_M}(M)$.   
\end{proof} 

Recall that the duality $D: \modd \cK_n\rightarrow\modd \cK_n; (M_i,M(\gamma_j))\mapsto (M^*_i, M^*(\gamma_j))$, where $M^*_i$ is the $k$-dual of $M_{3-i}$,  and $ M^*(\gamma_j)$ is the $k$-dual of $M(\gamma_j)$.  Then there exists a duality $D_{T(n)}: \modd(T(n),\Omega)\rightarrow\modd(T(n),\Omega)$, and we have $\pi_\lambda \circ D_{T(n)}(M)\cong D\circ \pi_\lambda(M)$ for all $M\in$ mod$(T(n),\Omega)$ \cite[Section 7.2]{Daniel}.  Using the duality $D_{T(n)}$, we can see that functor $D_{T(n)}^2$ is equivalent to the identity functor on the category  mod$(T(n),\Omega)$.

Since Lemma \ref{FK} tells us that we can take regular indecomposable module $M$ as a representation of subquiver $B_{+1}(M)$, we have $\sigma^-(D_{T(n)}M)= \cS^-_{1}\cdots \cS^-_{n_M}(D_{T(n)}M)$  by duality.  Then  we have the following result.

\begin{Cor}
Let $M\in\modd (T(n),\Omega)$ be a regular indecomposable module. Then 
\begin{center}
$\sigma^- (M)=\cS^-_{1'}\cdots \cS^-_{n'_M}(M)$.
\end{center}

\end{Cor}
\begin{proof}
This is a dual version of Lemma \ref{FK}.
\end{proof}

\begin{lemma}\label{FE}
Let $M\in \modd(T(n),\Omega)$ be a regular indecomposable module. Then
\begin{center}

$\sigma^-\sigma (M)=\cS^-_{n_M}\cdots \cS^-_1 \cS^+_1 \cdots \cS^+_{n_M}(M)\cong M$.
\end{center}
\end{lemma}
\begin{proof}
 Let $x\in (T(n),\Omega)_0\setminus B_{+1}(M)$ be a sink point with $y\in \cN(x)$, that is,  $x$ is a source point for $\sigma\Omega$.  Then $y$ is a sink point for $\sigma\Omega$. Moreover, the  point $y$ either satisfies  $\{y\}\cap B_{+1}(M)=\emptyset$ or $y$ is a leaf of $B_{+1}(M)$ by the choice of $x$.  Then we have  $\{y\}$ $ \cap$ supp$(M)=\emptyset$. Hence $(\sigma M)_y=M_y=0$. It indicates that $y\nin $ supp$(\sigma M)$. We have 
\begin{center}
 $\cN(x)$ $\cap $ supp$(\sigma M)=\emptyset$.
\end{center} 
Since $\sigma M$ is indecomposable by Lemma \ref{S1}, we have $x\nin$ supp$(\sigma M)$. It is obvious that $x\nin$ supp$(\sigma M)$ when $x$ is source point for $\Omega$. Then we have supp$(\sigma M)\subseteq B_{+1}(M)$ as sets of points, and $\cS^-_x(\sigma M)=\sigma M$ for any $x\in (T(n),\Omega)_0\setminus B_{+1}(M)$ being a sink point. Finally, we can see that $\sigma^-$ is the composition of those source points that contained in $B_{+1}(M)$ for $\sigma \Omega$, and they are exactly the points $1,\cdots,n_M$, whence we get our claim.

\end{proof}

Note that we can also get $B(\sigma M)\subseteq B_{+1}(M)$ as unoriented graphs from supp$(\sigma M)\subseteq B_{+1}(M)$  in the proof of  Lemma \ref{FE}.  Similarly,   we can get the following result.

\begin{Cor}\label{CI}
Let $M\in\modd (T(n),\Omega)$ be a regular indecomposable module. Then
\begin{center}
$\sigma \sigma^-(M)=\cS^+_{n'_M}\cdots \cS^+_{1'} \cS^-_{1'}\cdots \cS^-_{n'_M}(M)\cong M$.
\end{center}
\end{Cor}
\begin{proof}
This is a dual version of Lemma \ref{FE}.  That is, we can change $\sigma M$ into $\sigma^- M$ in the proof of  Lemma \ref{FE} and consider  $x\in  T(n)_0 \setminus B_{+1}(M)$ being a source point for $\Omega$, then $x$ is a sink point for $\sigma\Omega$. Comparing supp$(\sigma^- M)$ with $B_{+1}(M)$, we can finally get $\cN(x)$ $\cap$ supp$(\sigma^-M)=\emptyset$ and $x\nin$ supp$(\sigma^-M)$. Then supp$(\sigma^- M)\subseteq B_{+1}(M)$ and $\cS^+_x(\sigma^- M)=\sigma^- M$. Hence $\sigma$ is the composition of those sink points that contained in $B_{+1}(M)$ for $\sigma \Omega$, and they are exactly the points $1',\cdots,n'_M$, whence we get this. 

\end{proof}

 We also have  $B(\sigma^- M)\subseteq B_{+1}(M)$  from the proof of Corollary \ref{CI}.   According to Lemma \ref{FE} and Corollary \ref{CI},  we can take

\begin{center}
$M, \sigma M, \sigma^- M, \sigma^-\sigma (M)$ and $\sigma \sigma^-(M)$
\end{center}
as  representations of subquivers $(B_{+1}(M),\Omega)$  and $(B_{+1}(M),\sigma\Omega)$, respectively.

\begin{lemma}\label{S2}

There is a natural isomorphism between the functor
  $\sigma^- \circ \sigma$  and the identity functor. 

\end{lemma}
\begin{proof}
Let $1_{T(n)}$ denote the identity functor on the full subcategory of  $\modd(T(n),\Omega)$ consisting of regular modules. By Lemma \ref{S1}, we only need to talk about regular indecomposable modules.

Let $M\in$ mod$(T(n),\Omega)$ be a regular indecomposable module.  We now consider $M$ as a representation of subquiver $B_{+1}(M)$. According to Lemma \ref{FE},  there exists an isomorphism $\theta_M: M \rightarrow \sigma^-\sigma(M)$. Let $N\in\modd(T(n),\Omega)$ be regular indecomposable, and let $f:M\rightarrow N$ be a morphism. Let $B_{+1}(M,N)$ be the minimal ball containing $B_{+1}(M)$ and $B_{+1}(N)$. Then $B_{+1}(M,N)\subseteq T(n)_0$ is  finite.  It allows us to take $M,N$ as representations of subquiver $B_{+1}(M,N)$.
Let $F=\sigma^-\circ \sigma,$ and let $1,\cdots, n_{M,N}$ be the sink points of subquiver $B_{+1}(M,N)$. Consider the following diagram
\begin{center}
\begin{tikzcd}
M \arrow{r}{\theta_M} \arrow{d}{f}
& F(M)\arrow{d}{F(f)}\\
N \arrow{r}{\theta_N} & F(N)
\end{tikzcd}
\end{center}
Note that $B_{+1}(M), B_{+1}(N)\subseteq B_{+1}(M,N)$. By the proof of Lemma \ref{FE},   it is not hard to see that 

\begin{center}
 $F(X)\cong\cS^-_{n_{M,N}}\cdots \cS^-_1\cS^+_1 \cdots \cS^+_{n_{M,N}}(X)\cong X,$ 
\end{center}
where $X\in \{M,N\}$.  Moreover, we have 
\begin{center}
$\Hom_k(M,N)\cong \Hom_k(F(M),F(N))$.
\end{center}
Then there exists such $F(f)\in\Hom_k(F(M),F(N))$ so  that the above diagram commutes. Hence there exists a natural isomorphism $\theta: 1_{T(n)}\rightarrow F$. 
\end{proof}

\begin{Cor}\label{SN}
There is a natural isomorphism between the functor
  $\sigma \circ \sigma^-$  and the identity functor. 

\end{Cor}
\begin{proof}
 By Lemma \ref{S1}, we only need to talk about regular indecomposable modules.  Using Corollary \ref{CI}, we can finally get a dual proof with that of the Lemma \ref{S2}.

\end{proof}

\begin{lemma} \label{crf}
The functors  $\sigma$  and $ \sigma^-$ are quasi-inverses and induce    equivalences of the $k$-linear full subcategories of $\modd (T(n),\Omega)$ and   $\modd(T(n),\sigma\Omega)$ consisting of  regular modules.
\end{lemma}
\begin{proof}
Combining  Lemma \ref{S2} and Corollary \ref{SN}, the  functors $\sigma\circ \sigma^-$ and $\sigma^-\circ \sigma$ are equivalent to the identity functor on regular modules of mod$(T(n),\Omega)$, that is, $\sigma^-$, $\sigma$ are the inverse functors of $\sigma$, $\sigma^-$, respectively. Hence we have the claim. 

\end{proof}

\begin{lemma}\label{Tc}
 Let $M\in\modd (T(n),\Omega)$ be a regular indecomposable module. Then  we have $D_{T(n)}\circ \sigma(M)\cong \sigma^{-1}\circ D_{T(n)}(M)$ and $D_{T(n)}\circ \sigma^-(M)\cong \sigma\circ D_{T(n)}(M)$. 
\end{lemma}
\begin{proof}
 As a consequence of Lemma  \ref{crf}, we have these isomorphisms.
\end{proof}

\section{middle terms of AR-sequences of  T(n)}\label{sec4}

In this section, we discuss the  middle terms of  AR-sequences in the regular components of $\modd(T(n),\Omega)$.  Now let $M\in\modd(T(n),\Omega)$ be a regular indecomposable module.  Combining with  Theorem \ref{Theorem 1}, and sticking to its notation and conventions in \cite{Claus},   we define 
\begin{enumerate}
\item $\pi(M): =(a_0, \cdots, a_b)$ the \textit{center path} of $M$,
\item $p(M):=a_0$, $q(M):=a_b$. 
\end{enumerate}
 Note that  $p(M)$ is the center of all sink modules in the $\sigma$-orbit of $M$ and $q(M)$ is the center of all  its source modules, respectively. If, furthermore,  $b(M)$ is the number of flow modules in the $\sigma$-orbit of $M$,  then $b(M)=d(p(M), q(M))$.  Recall that $i=\iota(M)$.  By Theorem \ref{Theorem 1},  we can get the following corollary.
 
 \begin{corollary}\label{iff}
 Let  $M\in\modd (T(n),\Omega)$ be a regular indecomposable module. Then
 \begin{enumerate} 
\item $M$ is a sink module if and only if $i \leq 0$,

\item $M$ is a  flow module if and only if $0<i\leq b(M)$,

\item $M$ is a source module if and only if $i>b(M)$.
\end{enumerate}

 \end{corollary}

Let $ \cD$ be a regular component of  $\modd(T(n),\Omega)$,  and let $M,M'\in$  $\cD$ be two regular  indecomposable  modules.  Suppose that the $\sigma$-orbit of $M$ (or $M'$)  contains  $b$ (or $b'$) flow modules. According to \cite[Section 7]{Claus}, we have  $p(M)=p(M')$, $q(M)=q(M')$ and  $b=b'$. This implies that for the $\sigma$-orbits of $M$ and $M'$, the centers of all the sink modules or source modules are the same.  We also have  $r_0(M)-ql(M)=r_0(M')-ql(M')$, where $r_0(M)=r(M_0)$.  Then we define

\begin{center}
$p(\cD)=p(M)$,

$q(\cD)=q(M)$,

  $b(\cD)=d(p(M),q(M))$,

  $r(\cD)=r_0(M)-ql(M)$.
\end{center}
We have $p(M)=p(\sigma^t M)=p(M')=p(\sigma^{t'} M')$ and $q(M)=q(\sigma^t M)=q(M')= q(\sigma^{t'} M')$ for any $t, t' \in \mathbb{Z}$. That is,  $\pi(\sigma^t M)=\pi(\sigma^{t'} M')$ and $b(M)=b(\sigma^t M)=b(M')=b(\sigma^{t'} M')=b$.

 We say that a module $Y\in$ mod$(T(n),\Omega)$ is a \textit{sink} (\textit{source}, or \textit{flow}) module if  all its indecomposable direct summands are  sink (source, or  flow)  modules with the same center. Suppose that $Y=\oplus_{i\in I}Y^i$, where $Y^i$ is an indecomposable sink (source, or flow) module for all $i$, and $C(Y^i)=C(Y^j),i,j\in I$. Then we define $r(Y)=\max \{r(Y^i)\mid i\in I\}, C(Y)=C(Y^i)$. Let $0  \rightarrow X\rightarrow Y \rightarrow Z \rightarrow 0$ be an AR-sequence in a regular component $\mathcal{D}$. According to \cite[Lemma 2.2]{Claus1}, there exist at most two indecomposable direct summands in the middle term $Y$. Now suppose that  $Y=Y^1\oplus Y^2$, where $Y^1,Y^2\in\modd(T(n),\Omega)$ are indecomposable and $ql(Y^1)< ql(Y^2)$.  Then we have $b(Y^1)=b(Y^2)$ and $\iota(Y^1)=\iota(Y^2)$ [Theorem \ref{Theorem 3}].  In addition,  $C(Y^1)=C(Y^2)$ and $Y^1,Y^2$  are of the same type modules [Theorem \ref{Theorem 4}]. Hence we get $C(Y)=C(Y^2)$,  $r(Y)=r(Y^2) $. We define
\begin{center}
$\iota(Y): =\iota(Y^2)$ and  $b(Y): =b(Y^1)=b(Y^2)$. 
\end{center}
   In the following,  we focus on such middle term  $Y$ and let $Y=Y^2$ when $Y$ is indecomposable.

\begin{Lemma}\label{Lemma index}
Let $X, Y ,Z$ $\in\modd (T(n),\Omega)$ be regular modules.  Assume that
$0\rightarrow X\rightarrow Y\rightarrow Z\rightarrow 0$
is an AR-sequence. Then $\iota(Y)=\iota(\sigma Z)$. Moreover, $Y$ is a sink (flow, or source) module if and only if $\sigma Z$ is a sink (flow, or source) module.
 
\end{Lemma}
\begin{proof}
  We already know that  $\iota (Y^1)=\iota(Y^2)$  by Theorem \ref{Theorem 4}. We just need to prove $\iota(Y^2)=\iota(\sigma Z)$.  By   Theorem \ref{Theorem 3}, we get  $\iota(Z)=\iota(Y^2)-1$. Then  $\iota(\sigma Z)=\iota(Y^2)$.  This also tells us that   $Y$ and     $\sigma Z$  keep the same type [Corollary \ref{iff}].

\end{proof}

\begin{Theorem}\label{Lemma 3.3}
Let $\cD$ be a regular component of  $\modd (T(n),\Omega)$. Suppose that $0\rightarrow X\rightarrow Y\rightarrow Z\rightarrow 0$  is an AR-sequence in $\cD$.  Then  $C(Y)=C(\sigma Z)$ and $r(Y)=r(\sigma Z)+1$. 
 
\end{Theorem}

\begin{proof}

 We firstly assume that $\sigma Z$ is a sink module, then $Y$  and $Z$ would be  sink modules using Lemma \ref{Lemma index} and Proposition \ref{m sink}$(1)$. Then $C(Y)=p(Y)=p(Z)=C(Z)=C(\sigma Z)$. Module $\sigma^2 Z=X$ would be a source module whenever  modules $Y$ and $\sigma Z$ are both source modules  by Proposition \ref{3.3},  whence we still have $C(Y)=q(Y)=q(Z)=C(\sigma^2 Z)=C(\sigma Z)$ in this case.

By \cite[Proposition 1.$(b)$]{Claus}, the radii of all its flow modules  are  the same in the  $\sigma$-orbit of a regular indecomposable module. Moreover,   centers of those flow modules just follow the path from $p(M)$ to $q(M)$, that is, $\pi(M)$  [Proposition \ref{Flow}$(1)$].   However,   $\pi (Y)=\pi(Z)=\pi (\sigma Z)$ and $b(Y)=b(Z)=b(\sigma Z)$.  Hence  we must have $C(Y)=C(\sigma Z)$  when both of them are  flow modules. 
 
For the  second assertion, we use induction on the quasi-length of $Z$. Suppose that $ql(Z)=1$.  This implies that $Z$ is quasi-simple and $Y=Y^2$. In addition, $Z$ is  the quasi-top of  $Y$. Suppose that $\sigma Z$ is a sink module.   Then  $Y$ is a sink module  [Lemma \ref{Lemma index}]. By \cite[Lemma 3]{Claus}, we get   $B(Y)=B(Z)$, whence $Z$ is a sink and $r(Z)=r(\sigma Z)+1$ [Proposition \ref{m sink}$(1)$]. Then $r(Y)=r(Z)=r(\sigma Z)+1$. Suppose that    $\sigma Z$ is a source module. Then $Y$ is a source module and $X=\sigma^2 Z$ is the  quasi-socle of $Y$.  Then $B(Y)=B(\sigma^2 Z)$ \cite[Lemma 3$^*$]{Claus}. On the other hand, there exists $r(\sigma^2 Z)=r(\sigma Z)+1$, whence $r(Y)=r(\sigma Z)+1$ in this case using Proposition \ref{3.3}. Suppose that $\sigma Z$ is a flow module.  We get    $Y$ being a flow  module [Lemma \ref{Lemma index}].  But  we have seen that all flow modules have the same radius in the $\sigma$-orbit of a regular indecomposable module. Moreover,  $r(\sigma Z)=r(Z_0)-1$ and $r(Y)=r(Y_0)-1$.  Note that  $r_0(Y)-ql(Y)=r_0(Z)-ql(Z)$, $ql(Y_0)=ql(Y)$ and $ql(Z)=ql(Z_0)$ [Lemma \ref{crf}]. Since $ql(Y)=ql(Z)+1$, we have $r_0(Y)=r(Y_0)=r_0(Z)+1=r(Z_0)+1$. Finally, we get $r(Y)=r(\sigma Z)+1$.

Suppose that $ql(Z)=2$.  Then there exists a module $Z''\in \cD$ such that 
 
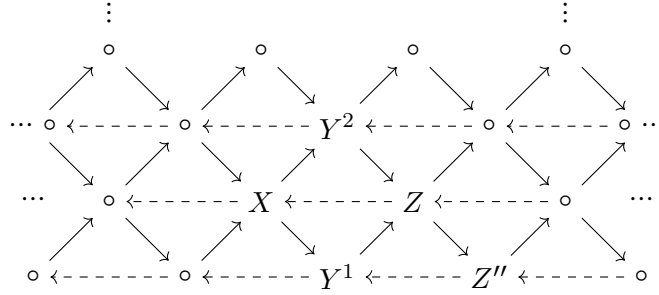
\begin{center}

\begin{tikzpicture}
\node (00) at (0,0) {$\circ$};
\node (02) at (0,2) {$\cdots$ $\circ$};
\node (11) at (1,1) {$\circ$};
\node (13) at (1,3) {$\circ$};
\node [above] at (13.north) {$\vdots$};
\node (01) at (0,1) {$\cdots$};
\node (81) at (8,1) {$\cdots$};
\path [->] (00) edge (11);
\path [->] (02) edge (11)  edge (13);

\node (20) at (2,0) {$\circ$};

\node (22) at (2,2) {$\circ$};

\path [->] (11) edge (20) edge (22);
\path [->] (13) edge (22) ;
\node (31) at (3,1) {$X$};

\node (33) at (3,3) {$\circ$};

\path [->] (20) edge (31);
\path [->] (22) edge (31)  edge (33);

\node (40) at (4,0) {$Y^1$};

\node (42) at (4,2) {$Y^2$};

\path [->] (31) edge (40) edge (42);
\path [->] (33) edge (42) ;
\node (51) at (5,1) {$Z$};

\node (53) at (5,3) {$\circ$};

\path [->] (40) edge (51);
\path [->] (42) edge (51)  edge (53);

\node (60) at (6,0) {$Z''$};

\node (62) at (6,2) {$\circ$};

\path [->] (51) edge (60) edge (62);
\path [->] (53) edge (62) ;
\node (71) at (7,1) {$\circ$};
\node (73) at (7,3) {$\circ$};
\node [above] at (73.north) {$\vdots$};
\path [->] (60) edge (71);
\path [->] (62) edge (71)  edge (73);

\node (80) at (8,0) {$\circ$};

\node (82) at (8,2) {$\circ$ $\cdots$};

\path [->] (71) edge (80) edge (82);
\path [->] (73) edge (82) ;

\path [->, dashed] (20) edge (00);
\path [->, dashed] (40) edge (20);
\path [->, dashed] (60) edge (40);
\path [->, dashed] (80) edge (60);
\path [->, dashed] (31) edge (11);
\path [->, dashed] (51) edge (31);
\path [->, dashed] (71) edge (51);
\path [->, dashed] (22) edge (02);
\path [->, dashed] (42) edge (22);
\path [->, dashed] (62) edge (42);
\path [->, dashed] (82) edge (62);

\end{tikzpicture}
\end{center}
We have $ql(Y^1)=ql(Z'')=1$ and $ql(Z)=2$.  The sequences $0\rightarrow Y^1\rightarrow Z\rightarrow Z''\rightarrow 0$  and $0\rightarrow \sigma Y^1\rightarrow \sigma Z\rightarrow \sigma Z''\rightarrow 0$ are  AR-sequences [Lemma \ref{crf}]. There are the equalities 
\begin{center}

$r(Z)=r(\sigma Z'')+1$,   and  $r(\sigma Z)=r(\sigma^2 Z'')+1$.

\end{center}
Hence we have 
\begin{center}
$r(Y)=r(Y^2)=r(Y^1)+2=r(\sigma^2 Z'')+2=r(\sigma Z)+1$
\end{center}
by Theorem \ref{Theorem 4}. Suppose that $ql(Z)> 2$. Similarly,  there exist modules $Z', W\in \cD$  such that

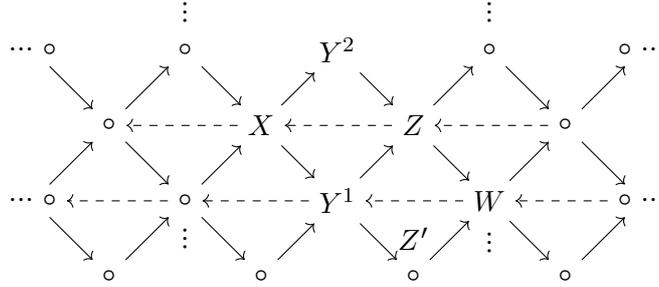
\begin{center}
\begin{tikzpicture}

\node (02) at (0,2) {$\cdots$ $\circ$};
\node (04) at (0,4) {$\cdots$ $\circ$};

\node (64) at (6,4) {$\circ$};
\node [above] at (64.north) {$\vdots$};

\node (11) at (1,1) {$\circ$};
\node (13) at (1,3) {$\circ$};

\path [->] (02) edge (11)  edge (13);
\path [->] (04) edge (13);

\node (22) at (2,2) {$\circ$};
\node [below] at (22.south) {$\vdots$};
\node (24) at (2,4) {$\circ$};
\node [above] at (24.north) {$\vdots$};
\path [->] (11)  edge (22);
\path [->] (13) edge (22) edge (24);
\node (31) at (3,1) {$\circ$};
\node (33) at (3,3) {$X$};

\path [->] (22) edge (31)  edge (33);
\path [->] (24) edge (33);

\node (42) at (4,2) {$Y^1$};

\node (44) at (4,4) {$Y^2$};

\path [->] (31) edge (42);
\path [->] (33) edge (42) edge (44);
\node (51) at (5,1) {$\circ$};
\node [above] at (51.north) {$Z'$};
\node (53) at (5,3) {$Z$};

\path [->] (42) edge (51)  edge (53);
\path [->] (44) edge (53);

\node (62) at (6,2) {$W$};
\node [below] at (62.south) {$\vdots$};

\path [->] (51) edge (62);
\path [->] (53) edge (62) edge (64);
\node (71) at (7,1) {$\circ$};
\node (73) at (7,3) {$\circ$};

\path [->] (62) edge (71)  edge (73);
\path [->] (64) edge (73);

\node (82) at (8,2) {$\circ$ $\cdots$};
\node (84) at (8,4) {$\circ$ $\cdots$};
\path [->] (71)edge (82);
\path [->] (73) edge (82) edge (84);
\path [->, dashed] (22) edge (02);
\path [->, dashed] (42) edge (22);
\path [->, dashed] (62) edge (42);
\path [->, dashed] (82) edge (62);
\path [->, dashed] (33) edge (13);
\path [->, dashed] (53) edge (33);
\path [->, dashed] (73) edge (53);

\end{tikzpicture}
\end{center}
In this case, we have $r(Y^1)=r(\sigma Z')+1$ by induction. However,  $r(Y^2)=r(Y^1)+2$ [Theorem \ref{Theorem 4}]. We get $r(Y)=r(\sigma Z')+3=r(\sigma Z)+1$ by $r(Z)=r(\sigma Z')+2$.
\end{proof}

\begin{proposition}
Let $\cD$ be a regular component of $\modd (T(n),\Omega)$. Then  $r(\cD)=r(\sigma \cD)$.

\end{proposition}

\begin{proof}
Let $M\in \cD$ be a regular indecomposable module. By the definition, we have $r(\cD)=r_0(M)-ql(M)$. However, $r_0(M)=r(M_0)$ and $ql(M)=ql(\sigma M)$ [Lemma \ref{crf}]. Hence $r(\sigma \cD)=r_0(\sigma M)-ql(\sigma M)=r(M_0)-ql(\sigma M)=r(\cD)$.

\end{proof}

\begin{proposition}\label{mt}
Let $X,  Y, Z \in\modd (T(n),\Omega)$ be regular  modules. Suppose that
\begin{center}
 $0\rightarrow X\rightarrow Y\rightarrow Z\rightarrow 0$
\end{center}
is an AR-sequence with $T(X)_0\cap T(Z)_0=\emptyset$. Then $Y$ is a flow module.
\end{proposition}
\begin{proof}
By Lemma \ref{Lemma index}, it is enough to prove that $\sigma Z $ is a flow module.
\begin{enumerate}

\item Suppose that $\sigma Z$ is a sink module with $C(\sigma Z)=x_0$, where $x_0\in T(\sigma Z)_0$.  According to  Proposition \ref{m sink}$(1)$,  $Z$ is a sink module and $C(Z)=x_0$.  For $\sigma^2 Z=X$, we have $C(X)=C(\sigma Z)=x_0$ when $X$ is a source module [Proposition \ref{m sink}$(3)$ ], and $C(X)\cap C(Z)=x_0$ when $X$ is a flow module [Proposition \ref{m sink}$(2)$ ].    Hence we always have  $C(Z)\cap C(X)=C(\sigma Z)\cap C(\sigma^2 Z)\neq \emptyset$ in this case. This  contradicts  our assumption $T(X)_0\cap T(Z)_0=\emptyset$.

\item Suppose that $\sigma Z$ is a source module with $C(\sigma Z)=y_0$,  $y_0\in T(\sigma Z)_0$. Then $\sigma^2 Z=X$ is a source module and $C(X)=C(\sigma Z)$ [Proposition \ref{3.3}]. Module $Z$ has 3 possibilites: sink, flow or source. By Proposition \ref{m sink}, \ref{Flow}, \ref{3.3}, we know that  $C(Z)= C(\sigma Z)=y_0$ when $Z$ is a sink or source module.  When $Z$ is a flow module, we have $C(Z)\cap C(\sigma Z)=y_0$, where $C(Z)=\{y'_0, y_0\}, y'_0\in T(n)_0$.  Since $T(X)_0\cap T(Z)_0=\emptyset$ and $C(X)\cap C(Z)=C(\sigma Z)\cap C(Z)\neq \emptyset$,  this case also cannot happen.
\end{enumerate}

Together with $(a)$ and $(b)$, we can see that $\sigma Z$ is a flow module.
\end{proof}

\begin{example}
Let $Z\in$ mod$(T(3),\Omega)$ be a sink module with the  following support,

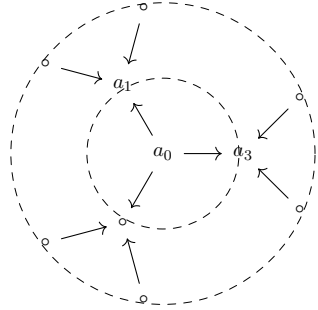
\begin{center}

\tikzstyle{level 1}=[sibling angle=120,level distance = 30, ->]
\tikzstyle{level 2}=[sibling angle=90,level distance = 30, <-]
\tikzstyle{level 3}=[sibling angle=60,level distance = 30, ->]
\tikzstyle{level 4}=[sibling angle=45,level distance = 30, <-]
\tikzstyle{every node}=[]
\tikzstyle{edge from parent}=[segment angle=10,draw]
\begin{tikzpicture}[grow cyclic, shape=circle,cap=round, scale=1, every node/.style={scale=0.7}]
\node {$a_0$} 
    child { node {$\circ$} child {node {$\circ$}  } child {node {$\circ$} }}
    child { node {$a_3$} child {node {$\circ$}  } child {node {$\circ$}}}
    child { node {$a_1$} child {node {$\circ$}  } child {node {$\circ$} }};
\draw [dashed, line width=0.1pt]circle(1cm);
\draw [dashed, line width=0.1pt]circle(2cm);

\end{tikzpicture}

\end{center}
 where dim$_k Z_a=1, a\in T(Z)_0$ and $Z(\gamma_i)=\lambda_i,\lambda_i\in k\setminus\{0\},i\in\{1,3\}$.  Suppose that $0\rightarrow X\rightarrow Y\rightarrow Z\rightarrow 0$ is an AR-sequence. Then support of $X=\sigma^2Z$ would look like the following.
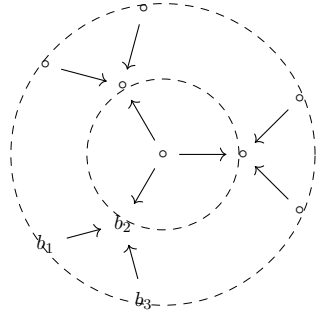
\begin{center}

\tikzstyle{level 1}=[sibling angle=120,level distance = 30, ->]
\tikzstyle{level 2}=[sibling angle=90,level distance = 30, <-]
\tikzstyle{level 3}=[sibling angle=60,level distance = 30, ->]
\tikzstyle{level 4}=[sibling angle=45,level distance = 30, <-]
\tikzstyle{every node}=[]
\tikzstyle{edge from parent}=[segment angle=10,draw]
\begin{tikzpicture}[grow cyclic, shape=circle,cap=round, scale=1, every node/.style={scale=0.7}]
\node {$\circ$} 
    child { node {$b_2$} child {node {$b_1$} } child {node {$b_3$} }}
    child { node {$\circ$} child {node {$\circ$} } child {node {$\circ$} }}
    child { node {$\circ$} child {node {$\circ$} } child {node {$\circ$} }};
\draw [dashed, line width=0.1pt]circle(1cm);
\draw [dashed, line width=0.1pt]circle(2cm);

\end{tikzpicture}

\end{center}
Hence $X$ is a source module. Moreover, we have $T(X)_0\cap T(Z)_0=\emptyset$.  Let $(a_1,a_0,a_3)$ be a diameter path of $T(Z)$. By direct computation,  module $\sigma Z: k \xrightarrow{\gamma_2} k,\gamma_2\in k\setminus\{0\}$  is a flow module. Then $C(\sigma Z)=\{a_0,b_2\}$ and $r(\sigma Z)=0$.  According to Lemma \ref{Lemma index} or Proposition \ref{mt},   $Y$ is a flow module.  By Theorem \ref{Lemma 3.3}, we have $r(Y)=r(\sigma Z)+1=1$  and $C(Y)=C(\sigma Z)=\{a_0,b_2 \}$.

\end{example}

 \section*{Declarations}

 \textbf{Ethical Approval}
 
  This declaration is not applicable.
\vspace{0.2cm}

 \textbf{Competing interests}
 
  This declaration is not applicable.
\vspace{0.2cm}

 \textbf{Authors' contributions}
 
  This declaration is not applicable.

\vspace{0.2cm}

 \textbf{Funding}
 
  This declaration is not applicable.

\vspace{0.2cm}

 \textbf{Availability of data and materials}
 
  This declaration is not applicable.


\end{document}